\documentclass[a4paper,11pt]{article}

\usepackage[pagewise]{lineno}\nolinenumbers

\usepackage{amsmath,amssymb,mathrsfs}
\usepackage{amsthm}
\usepackage{dsfont}
\usepackage{graphicx}
\usepackage{empheq}
\usepackage{faktor}
\usepackage{enumitem}
\usepackage{hyperref}
\usepackage{mathtools}
\usepackage{etoolbox}

\usepackage{blkarray}

\usepackage{tikz-cd}
\tikzcdset{ampersand replacement=\&}
\tikzcdset{row sep/normal=2cm}
\tikzcdset{column sep/normal=2.5cm}

\usepackage{geometry}
\usepackage[english]{babel}
\usepackage[T1]{fontenc}
\usepackage[utf8]{inputenc}

\DeclareFontFamily{U}{matha}{\hyphenchar\font45}
\DeclareFontShape{U}{matha}{m}{n}{
<-6> matha5 <6-7> matha6 <7-8> matha7
<8-9> matha8 <9-10> matha9
<10-12> matha10 <12-> matha12
}{}
\DeclareSymbolFont{matha}{U}{matha}{m}{n}

\DeclareFontFamily{U}{mathx}{\hyphenchar\font45}
\DeclareFontShape{U}{mathx}{m}{n}{
<-6> mathx5 <6-7> mathx6 <7-8> mathx7
<8-9> mathx8 <9-10> mathx9
<10-12> mathx10 <12-> mathx12
}{}
\DeclareSymbolFont{mathx}{U}{mathx}{m}{n}

\DeclareMathDelimiter{\vvvert} {0}{matha}{"7E}{mathx}{"17}%

\DeclarePairedDelimiterX{\normiii}[1]
{\vvvert}
{\vvvert}
{\ifblank{#1}{\:\cdot\:}{#1}}

\usepackage[backend=biber,natbib=true,style=numeric,doi=false]{biblatex}

\newcommand{\dis}{\displaystyle}

\newcommand{\p}{\partial}
\newcommand{\beq}{\begin{equation}}
\newcommand{\eeq}{\end{equation}}
\renewcommand{\epsilon}{\varepsilon}

\renewcommand{\leq}{\leqslant}
\renewcommand{\geq}{\geqslant}
\renewcommand{\d}{\mathrm{d} }
\newcommand{\R}{\mathbb{R}}
\newcommand{\N}{\mathbb{N}}

\newcommand{\ddt}{\frac{\p}{\p t}}
\newcommand{\supp}{\mathrm{supp \,}}

\newcommand{\M}{\mathcal{M}}
\newcommand{\A}{\mathcal A}
\newcommand{\e}{\mathrm e}

\newcommand{\B}{\mathfrak{B}}

\newcommand{\1}{\mathds 1}
\newcommand{\Cr}{\mathcal{C}}
\newcommand{\muin}{\mu^{\mathrm{in}}}

\newtheorem{defi}{Definition}
\newtheorem{theorem}{Theorem}
\newtheorem*{theorem*}{Theorem}
\newtheorem{prop}[theorem]{Proposition}
\newtheorem{lem}[theorem]{Lemma}
\newtheorem{coro}[theorem]{Corollary}
\newtheorem*{remark}{Remark}

\title{Asymptotic behaviors of a kinetic approach to the collective dynamics of a rock-paper-scissors binary game}

\author{Hugo Martin \thanks{Universit\'e Rennes 1 $\&$ Institut d'Agro, Rennes, France. Email: hugo.martin1@polytechnique.edu}}

\begin{document}

\maketitle

\begin{abstract}
This article studies the kinetic dynamics of the rock-paper-scissors binary game in a measure setting given by a non local and non linear integrodifferential equation. After proving the wellposedness of the equation, we provide a precise description of the asymptotic behavior in large time. To do so we adopt a duality approach, which is well suited both as a first step to construct a measure solution by mean of semigroups and to obtain an explicit expression of the asymptotic measure. Even thought the equation is non linear, this measure depends linearly on the initial condition. This result is completed by a decay in total variation norm, which happens to be subgeometric due to the nonlinearity of the equation. This relies on an unusual use of a confining condition that is needed to apply a Harris-type theorem, taken from a recent paper~\cite{Canizo2021} that also provides a way to compute explicitly the constants involved in the aforementioned decay in norm.
\end{abstract}

\

\noindent{\bf Keywords:} kinetic equations, binary games, measure solutions, nonlinearity, long-time behavior, explicit limit, subgeometric convergence rate.

\

\noindent{\bf MSC 2010:} Primary: 45K05, 35B40, 35R06; Secondary: 91A05 

\section{Introduction}

In a recent article~\cite{Duteil2021}, Pouradier Duteil and Salvarani introduced a kinetic equation to describe a large population of agents that interact by mean of random encounters and wealth exchange. Transcient pairs are formed with probability $\eta$, then if both of the agents are rich enough (in a sense to be clarified below), they play a game of rock-paper-scissors, and based on the result possibly exchange money. If the outcome of the game is a draw, then the transcient pair is unmade without change in the players' wealth. Otherwise, the winner receives a fixed quantity $h>0$ from the other player. The payoff matrix of player 1 for this zero-sum rock-paper-scissors game is given by
\[
\begin{blockarray}{ccccc}
&  & R & P & S \\
\begin{block}{cc(ccc)}
  R & & 0 & h & -h  \\
  P & & -h & 0 & h  \\
  S & & h & -h & 0  \\
\end{block}
\end{blockarray}
 \]
Standard results from game theory state that the optimal strategy is in this case a mixed strategy and a Nash equilibrium, that is selecting at random and uniformly one of the three moves \cite{Hofbauer1998}.
\\
The population of players, structured in wealth, is described by a distribution function $u=u(t,y)$ defined on $\R_+^2$. Given a subset $A\subset\R_+$, the integral
\[
\int_A u(t,y)\d y
\]
represents the number of individuals whose wealth belongs to $A$. Denoting $u^\mathrm{in}$ the initial wealth distribution, this model reads
\begin{equation}
\left\{\begin{array}{l}
\dis\p_t u(t,y) = \frac{\eta}{3}\int_h^\infty u(t,y')\d y' \left[u(t,y-h)\1_{[2h,\infty)}(y) + u(t,y+h)\1_{[0,\infty)}(y) -2u(t,y)\1_{[h,\infty)}(y)\right]\label{eq:pde} \vspace{4mm}\\
u(0,y) = u^\mathrm{in}(y), \qquad \forall\, y\geq 0.
\end{array}\right.
\end{equation}
The intensity of the exchange between players is proportional to the integral
\[\dis \int_h^\infty u(t,y')\d y'\]
which makes this equation non linear. This term comes from the rule that forbids debts, implying that players involved  in a transcient pair only play the game if they own at least $h$.
\\
The authors of the aforementionned article provided (among other results) wellposedness of this equation in a $\mathrm{L}^1$ setting as well as its behavior as $h$ vanishes under the diffusive scalling $t\leftarrow t/h^2$. The large time asymptotics when $h$ remains fixed was yet to be investigated: so is the purpose of the present paper. Our goal is to provide a precise asymptotic behavior to this equation, drawing inspiration from the methodology developped in~\cite{Gabriel2022} for a critical case of the growth-fragentation equation. In this article, the authors worked in a measure framework and adopted a combination of semigroup and duality approach. They obtained both an uniform exponential decay using the results from~\cite{Hairer2011} and a formula for the invariant probability measure, that was explicit in term of direct and adjoint eigenvector of the growth-fragmentation equation (see~\cite{DG} for a rich survey as well as very general assumptions ensuring the existence of such functions for this equation). This equation is non local yet linear, unlike the one studied in the present article which is both non local and non linear. Such feature can arguably be considered as the hallmark of the lost of exponential relaxation toward a stationary solution, that is indeed verified in the present case, see Theorem~\ref{thm:main} below. Taking advantage of the particular expression of Equation~\eqref{eq:pde}, an appropriate time rescaling enables to use results from the recent paper~\cite{Canizo2021} from Ca\~nizo and Mischler.
\\
The remaining of the paper is organized as follows. In the next section, we introduce the framework that is required to solve Equation~\eqref{eq:pde} in the sense of measures and state our main result. Our methodoloy relies on a duality approach, so Section~\ref{sec:dual} is devoted to the study of the adjoint equation. In Section~\ref{sec:measure}, we build on previous results to construct a measure solution to Equation~\eqref{eq:pde}. Section~\ref{sec:asympt} contains a precise description of the asymptotic behavior of the solution. These results are illustrated by numerical simulations in Setion~\ref{sec:num}. Finally in Section~\ref{sec:con}, we propose some possible continuations of this paper.

\section{Preliminaries and the main result}

We start by recalling briefly the notions from measure theory that we need to establish our results. For a more complete introduction on this field, we refer the reader to~\cite{Rudin1998} in which the focus is on the total variation norm, and the recent book~\cite{Duell2021} for a rich exposition on measure solutions to PDEs, in particular using the topology of the \textit{flat norm} or \textit{dual bounded Lipschitz norm}.
\\
We endow $\R_+ = [0,+\infty)$ with its standard topology and the associated Borel $\sigma-$algebra $\B(\R_+)$. Throughout the paper, we shall consider discrete subsets of $\R_+$ and unions of such sets, so for a subset $\Omega \subset \R_+$,  we denote by
$\M(\Omega)$ the space of real-valued Radon measures with Hahn-Jordan decomposition $\mu = \mu_+ - \mu_-$ on $\Omega$ such that 
\[
\|\mu\| := \int_\Omega \d |\mu| < \infty
\]
where $|\mu| = \mu_+ + \mu_-$ is the total variation measure of $\mu$. To obtain the desired asymptotic results, we shall work in spaces of weighted measures. For a measurable (weight) function $V : \Omega \to [1, \infty)$, we denote by $\M_V(\Omega)$ the subspace of finite signed measures $\mu$ on $\Omega$ such that
\[
\|\mu\|_V = \int_\Omega V\d |\mu| < \infty,
\]
and simply $\|\cdot\|$ whenever $V\equiv 1$.
\\
Now we denote $\mathcal{B}_V(\Omega)$ the space of Borel functions $f : \Omega \to \R_+$ such that
\[\|f\|_{\mathcal{B}_V(\Omega)} := \sup_{y\in\Omega}\frac{|f(y)|}{V(y)}  < \infty.\] If a function $f\in\mathcal{B}_V(\Omega)$ is also continuous, we denote $f\in\mathcal{C}_V(\Omega)$. In the case $V \equiv 1$, we simply write $\|\cdot\|_\infty$ for this norm, and more generally omit the index $V$.
For every $\mu\in\M_V(\Omega)$, one can define a linear form on $\mathcal{B}_V(\Omega)$ through the duality bracket
\[
f \mapsto \langle \mu,f\rangle := \int_\Omega f \d\mu.
\]
With a slight abuse of notation, for a measurable set $A$, we will write $\mu(A)$ instead of $\langle \mu,\mathds{1}_A\rangle$. The norm $\|\cdot\|_V$ can be expressed as
\[
\|\mu\|_{V} = \sup_{\|f\|_{\mathcal{B}_V(\Omega)} \leq 1} \langle \mu,f\rangle.
\]
\\
Now we define a weaker norm on the space of measures. First, for a function $f$ continuous on $\Omega$, we define
\[|f|_{Lip} := \sup_{y\neq z}\frac{|f(y) - f(z)|}{|y-z|}.\] Then we can define the Lipschitz bounded norm
\[\|f\|_{BL(V)} := \|f\|_{\mathcal{B}_V(\Omega)} + |f|_{Lip}.\] The dual bounded Lipschitz norm is thus defined as
\[\|\mu\|_{BL^*(V)} := \sup \left\{\int_\Omega f\d \mu : f\in\mathcal{C}(\Omega),\, \|f\|_{BL(V)}\leq 1\right\}\]
with $\mathcal{C}(\Omega)$ denoting the set of continuous functions on $\Omega$. Since the supremum is taken on a smaller set, it is clear that for any measure $\mu$, one has $\|\mu\|_{BL^*(V)} \leq \|\mu\|_{V}$. In particular, one has $\|\delta_y - \delta_z\| = 2$ if $y\neq z$ but $\|\delta_y - \delta_z\|_{BL^*(1)} = \min(1,|y-z|)$, so $(\M(\Omega),\|\cdot\|_{BL^*})$ enjoys better topological properties that $(\M(\Omega),\|\cdot\|_{TV})$.
\\
It remains to define a notion of measure solutions to Equation\eqref{eq:pde}. We choose an equation of the "mild" type, in the sense that it relies on an integration in time. Let us give our motivation for this choice. Assume that $u(t,y) \in
\Cr^1(\R_+; \mathrm{L}^1(\R_+))$ satisfies~\eqref{eq:pde} in the classical sense. Integrating this equation multiplied by $f\in\mathcal{B}(\R_+)$, we obtain
\begin{align*}
\int_0^\infty f(y)u(t,y)&\d y = \int_0^\infty f(y)u^{\mathrm{in}}(y)\d y \\
+& \int_0^t \int_0^\infty \frac{\eta}{3}\left(\int_h^\infty u(z,s)\d z\right)\left[f(y+h) + f(y-h) - 2f(y)\right]\mathds{1}_{[h,\infty)}(y)u(s,y)\d y\, \d s.
\end{align*}
In order to state the equivalent of this equation for a general signed measure, we define the operator $\A$ on $\mathcal{B}(\R_+)$ by
\[
\A f : y\mapsto \left[f(y+h) + f(y-h) - 2f(y)\right]\mathds{1}_{[h,\infty)}(y).
\]
\begin{defi}\label{def:sol}
A family $(\mu_t)_{t\geq 0}\subset \M(\R_+)$ with initial condition $\muin$ is called a measure solution to Equation~\eqref{eq:pde} if for all $f\in\mathcal{B}(\R_+)$ one has
\begin{equation}\label{eq:meas_const}
\langle\mu_t, f\rangle = \langle\muin, f\rangle + \int_0^t \left\langle\mu_s, \frac{\eta}{3}\mu_s([h,\infty))\A f\right\rangle\d s.
\end{equation}
\end{defi}
We introduce now another slight abuse of notation. For a measurable set $A\subset \R_+$ and a positive real number $a$, we denote
\[A + a := \left\{y\in \R_+ : \exists x\in A,\, y = x + a\right\}.\]
Finally, we are ready to state the main result of this paper, which is about the wellposedness of Equation~\eqref{eq:pde} in the measure setting, as well as the asymptotic bahaviour of the solution, when the exchange parameter $h$ remains fixed.
\begin{theorem}\label{thm:main}
For any initial condition $\muin \in \M(\R_+)$, there exists a unique measure solution $(\mu_t)_{t\geq 0}$ to Equation~\eqref{eq:pde} in the sense of Definition~\ref{def:sol}, a projection operator $\mathcal{P}_h$ defined on $\M(\R_+)$ and constants $C,\lambda > 0$ independant of $\eta$ and $\muin([h,\infty))$ such that
\begin{equation}\label{decay}
\forall t\geq 0, \qquad\left\|\mu_t - \muin \mathcal{P}_h\right\| \leq \frac{C}{\left(1+\frac{\eta\muin([h,\infty))}{3}t\right)^\lambda}\left\|\muin - \muin \mathcal{P}_h\right\|.
\end{equation}
The constants $C$ and $\lambda$ can be computed explicitely, and the measure $\muin \mathcal{P}_h$ is given explicitely in terms of the initial condition $\muin$ and the exchange parameter $h$:
\begin{itemize}
\item $\supp \muin \mathcal{P}_h \subset [0,h)$
\item for all measurable set $A\subset [0,h)$
\[\muin \mathcal{P}_h(A) = \sum_{k=0}^\infty \muin\left(A + kh\right).\]
\end{itemize}
\end{theorem}
\begin{remark}
A particular feature of Equation~\eqref{eq:pde} is the conservation of the population and total wealth.
Indeed, a formal integration against the measures $\d y$ and $y \d y$ over $[0, \infty)$ leads to the
balance laws
\[\frac{\d}{\d t}\int_0^\infty u(t,y)\d y = \frac{\d}{\d t}\int_0^\infty yu(t,y)\d y = 0.\]
In the language of measures, this translates by $\mu_t(\R_+) = \muin(\R_+)$ and $\langle \mu_t,Id\rangle = \langle \muin,Id\rangle$ with $Id$ the identity function, provided $\muin$ as a finite first moment. If $(\mu_t)_{t\geq 0}$ is a measure solution, then the first equality is satisfied by definition, since $\mathds{1}_{\R_+}$ lies in $\mathcal{B}(\R_+)$. If in addition $\langle \muin,Id\rangle$ is finite, one can define the formula~\eqref{eq:meas_const} on test functions $f\in\mathcal{B}_{1+Id}(\R_+)$, \textit{i.e.} functions that are Borel and satisfy
\[
\sup_{y\in\Omega} \frac{|f(y)|}{1+y} < \infty.
\]
The function $Id$ lies in this set and satisfies $\A f = 0$. Thus the equality $\langle \mu_t,Id\rangle = \langle \muin,Id\rangle$ is satisfied, meaning that the total wealth is conserved at any finite time $t$.
\end{remark}
Now, we make a statement about the behavior of the `asymptotic in time' measure $\muin\mathcal{P}_h$ when $h$ vanishes.
\begin{theorem}
Denoting $\mathcal{P}_0$ the linear operator acting on measures defined by $\mu \mathcal{P} := \mu(\R_+)\delta_0$, one has
\[\normiii{\mathcal{P}_h - \mathcal{P}_0}_{BL^*} \leq h\]
with $\normiii{\cdot}_{BL^*}$ the operator norm on $(\M(\R_+),\|\cdot\|_{BL^*})$ defined by
\[\normiii{\mathcal{P}}_{BL^*} := \sup_{\|\mu\|_{BL^*} \leq 1} \frac{\|\mu\mathcal{P}\|_{BL^*}}{\|\mu\|_{BL^*}}.\]
\end{theorem}


\section{Dual equation and right semigroup\label{sec:dual}}

This section is devoted to the study of the wellposedness of a family of equations that are related to the adjoint equation. Assume $(\mu_t)_{t\geq 0}$ is the unique solution to~\eqref{eq:meas_const}. Then, the (backward) adjoint equation is given by
\begin{align}\label{eq:adj_const}
\ddt f(t,s,y) & = \frac{\eta}{3}\mu_t([h,\infty))\A f(t,s,y)\nonumber\\
& = \frac{\eta}{3}\mu_t([h,\infty))\left[f(t,s,y+h) + f(t,s,y-h) - 2f(t,s,y)\right]\mathds{1}_{[h,\infty)}(y),
\end{align}
with $(t,s,y)\in \R_+^3$ and $s\leq t$, supplemented with the terminal condition $f(t,t,y) = f_0(y)$. The second argument $s$ is included to account for the inhomogeneity in the time evolution, since the solution of~\eqref{eq:adj_const} depends on the values of $t\mapsto \mu_t([h,\infty))$. Due to the indicator function $\mathds{1}_{[h,\infty)}$, a solution $f$ to Equation~\eqref{eq:adj_const} satisfies $f(t,s,y) = f_0(y)$ for all $y\in [0,h)$ and $0\leq s \leq t$. On the interval $[h,\infty)$, we solve a \text{mild} version of this equation, that is obviously more complicated than on $[0,h)$. All in all, we search for a function $f$ that satisfies
\begin{align}\label{eq:mild}
f(t,s,y) =& f_0(y)\e^{-\frac{2\eta}{3}\mathds{1}_{[h,\infty)}(y)\int_s^t \mu_\sigma([h,\infty)) \d \sigma}\nonumber\\
 &\qquad + \frac{\eta}{3}\mathds{1}_{[h,\infty)}(y)
\int_s^t \mu_\sigma([h,\infty)) \e^{-\frac{2\eta}{3}\int_\sigma^t \mu_\tau([h,\infty)) \d \tau}[f(\sigma,s,y + h) + f(\sigma,s,y - h)]\d \sigma
\end{align}
for all $y>0$. In this form, this equation depends on the solution $(\mu_t)_{t\geq 0}$ of the direct problem, so we introduce related equations by replacing $\mu_\bullet([h,\infty))$ by a generic non negative and continuous $b$ function. The equation we study is then
\begin{align}\label{eq:mild_const}
f(t,s,y) =& f_0(y)\e^{-\frac{2\eta}{3}\mathds{1}_{[h,\infty)}(y)\int_s^t b(\sigma) \d \sigma}\nonumber\\
 &\qquad + \frac{\eta}{3}\mathds{1}_{[h,\infty)}(y)
\int_s^t b(\sigma) \e^{-\frac{2\eta}{3}\int_\sigma^t b(\tau) \d \tau}[f(\sigma,s,y + h) + f(\sigma,s,y - h)]\d \sigma.
\end{align}
In this section, we will solve Equation~\eqref{eq:mild_const} for a fixed function $b$. The wellposedness of this problem, as well as useful properties, are collected in the next proposition. First, we introduce some notations. For $\tilde{\Omega} \subset \R_+^2$ and $\Omega \subset \R_+$, we denote $\Cr\left(\tilde{\Omega},\mathcal{B}(\Omega)\right)$ the set of functions $f$ that are defined and continuous on $\tilde{\Omega}$ such that for all $(t,s)\in \tilde{\Omega}$, the function $f(t,s,\cdot)$ lies in $\mathcal{B}(\Omega)$. Similarly, we define a subset of the previous one, denoted $\Cr^1\left(\tilde{\Omega},\mathcal{B}(\Omega)\right)$ made of the functions such that $\p_t f$ and $\p_s f$ lie in $\Cr\left(\tilde{\Omega},\mathcal{B}(\Omega)\right)$. 
For $x\in[0,h)$ we define%
\[\mathfrak{C}_x := \left\{x+kh, k\in\N\right\}.\]
Since players can only gain or loose $h$ after each game, they shall remain in the same `class of wealth' at all time, depending on their initial wealth only. Such classes are precisely these sets $\mathfrak{C}_x$ for $x\in[0,h)$. Since any $f\in \mathcal{B}(\R_+)$ can be written as
\[
f = \sum_{x\in [0,h)} f_{|\mathfrak{C}_x}
\]
it is enough to prove properties on $\mathcal{B}(\Omega\times \mathfrak{C}_x)$, which we do for the next result.
\begin{prop}\label{prop:adj}
For all $f_0 \in \mathcal{B}(\mathfrak{C}_x)$ and function $b$ continuous, positive and non increasing, there exists a unique solution $f_b$ to Equation~\eqref{eq:mild_const} in $\Cr(\R_+^2,\mathcal{B}(\mathfrak{C}_x)) \cap \Cr^1([0,T]^2,\mathcal{B}(\mathfrak{C}_x))$ for any $T>0$.
Additionally
\begin{itemize}
\item[] \qquad for all $(t,s)\in \R_+^2$ and $x\in[0,h)$, $\|f_b(t,s,\cdot)\|_{\mathcal{B}(\mathfrak{C}_x)} \leq \|f_0\|_{\mathcal{B}(\mathfrak{C}_x)}$;
\item[] \qquad if $f_0 \geq 0$ then for all $t \geq s\geq 0$, $f_b(t,s,\cdot)\geq 0$;
\item[] \qquad if $f_0 = \mathds{1}_{\mathfrak{C}_x}$, then $f_b(t,s,\cdot) = \mathds{1}_{\mathfrak{C}_x}$ for all $(t,s)\in \R_+^2$,
\item[] \qquad if $b$ is non increasing, then for all $(t,s)\in \R_+^2$ and $x\in[0,h)$,\\ $\|f_b(t,s,\cdot)\|_{\Cr^1_x} \leq \left(\|f_0\|_{\mathcal{B}(\mathfrak{C}_x)} + 2b(0)\|\A f_0\|_{\mathcal{B}(\mathfrak{C}_x)}\right)$ and $f_b$ lies in $\Cr^1(\R_+^2,\mathcal{B}(\mathfrak{C}_x))$, with \\
\[
\|f\|_{\mathcal{C}^1_x} := \|f\|_{\mathcal{B}(\mathfrak{C}_x)} + \|\p_t f\|_{\mathcal{B}(\mathfrak{C}_x)} + \|\p_s f\|_{\mathcal{B}(\mathfrak{C}_x)}.
\]
\end{itemize}
\end{prop}
\begin{proof}
Fix $b \in \Cr(\R_+)$ with the aforementionned properties and $T >0$. Let $\Gamma$ be the operator defined on $\mathcal{B}([0,T]^2\times\mathfrak{C}_x)$
\begin{align*}
\Gamma f(t,s,y) =& f_0(y)\e^{-\frac{2\eta}{3}\mathds{1}_{[h,\infty)}(y)\int_s^t b(\sigma) \d \sigma}\\
 &\qquad + \frac{\eta}{3}\mathds{1}_{[h,\infty)}(y)\int_s^t b(\sigma) \e^{-\frac{2\eta}{3}\int_\sigma^t b(\tau) \d \tau}[f(\sigma,s,y + h) + f(\sigma,s,y - h)]\d \sigma,
\end{align*}
with $(t,s,y) \in [0,T]^2\times \mathfrak{C}_x$, $T>0$ and $f_0\in \mathcal{B}(\mathfrak{C}_x)$. Considere functions \[f,g\in \left\{\varphi\in\mathcal{B}([0,T]^2\times\mathfrak{C}_x),\, \forall s\in [0,T],\, \varphi(s,s,\cdot) = f_0\right\}.\] We easily show that
\[\sup_{(t,s)\in [0,T]^2}\|\Gamma f(t,s,\cdot) - \Gamma g(t,s,\cdot)\|_{\mathcal{B}(\mathfrak{C}_x)} \leq \frac{2\eta}{3}\|b\|_\infty T\sup_{(t,s)\in [0,T]^2}\|f(t,s,\cdot)-g(t,s,\cdot)\|_{\mathcal{B}(\mathfrak{C}_x)}\] so for $T<\frac{3\|b\|_\infty}{2\eta}$, the operator $\Gamma$ is a contraction and thus has a single fixed point on $[0,T]$. To prove the claim on the boundedness of the solution, we prove that the closed ball of radius $\|f_0\|_{\mathcal{B}(\mathfrak{C}_x)}$ is invariant under $\Gamma$. One has $\Gamma f(t,s,x) = f_0(x)$, so
\[
\sup_{(t,s)\in [0,T]^2}\left|\Gamma f(t,s,x)\right| \leq \|f_0\|_{\mathcal{B}(\mathfrak{C}_x)}.
\]
Now for $y>x$, we compute
\[
\left|\Gamma f(t,s,y)\right| \leq \|f_0\|_{\mathcal{B}(\mathfrak{C}_x)}\e^{-\frac{2\eta}{3}\int_s^t b(s) \d s} + \frac{2\eta}{3}\int_s^t b(\sigma)\e^{-\frac{2\eta}{3}\int_\sigma^t b(\tau) \d \tau}\|f(\sigma,s,\cdot)\|_{\mathcal{B}(\mathfrak{C}_x)}\d\sigma
\]
so \[\sup_{(t,s)\in [0,T]^2}\|f(t,s,\cdot)\|_{\mathcal{B}(\mathfrak{C}_x)} \leq \|f_0\|_{\mathcal{B}(\mathfrak{C}_x)} \qquad \Longrightarrow \qquad \sup_{(t,s)\in [0,T]^2}\|\Gamma f(t,s,\cdot)\|_{\mathcal{B}(\mathfrak{C}_x)} \leq \|f_0\|_{\mathcal{B}(\mathfrak{C}_x)},\] thus the claim is proved. As a consequence, one can iterate the fixed point procedure on time intervals of length $T$ \textit{ad nauseam}, to finally obtain a unique global solution defined on $\R^3_+$.
\\
If $f_0$ is nonnegative, the operator $\Gamma$ preserves the closed cone of Borel functions defined on $\{(t,s,y)\in \R^2_+\times\mathfrak{C}_x, s\leq t\}$, so the fixed point $f_b$ satisfies $f_b(t,s,\cdot)\geq 0$ for all $t\geq s \geq 0$. In addition, if $f_0 = \mathds{1}_{\mathfrak{C}_x}$, we easily obtain by computations that $\mathds{1}_{\mathfrak{C}_x}$ is the fixed point. 
\\
Then, we prove that the time derivative of the aforementioned solution $f$ lies in $\Cr^1([0,T]^2,\mathcal{B}(\mathfrak{C}_x))$ for any $T>0$. Let $f,g\in\left\{\varphi\in\Cr^1([0,T]^2,\mathcal{B}(\mathfrak{C}_x)),\, \forall s\in [0,T],\, \varphi(s,s,\cdot) = f_0\right\}$. This set is $\Gamma-$invariant if $f_0$ lies in $\mathcal{B}(\mathfrak{C}_x)$. For $t,s\geq0$ and $y \geq h$ we compute the derivative in the first variable and obtain after integrating by parts
\begin{align}
\p_t \Gamma f(t,s,y)& = b(t)\A f_0(y)\mathds{1}_{[h,\infty)}(y)\e^{-\frac{2\eta}{3}\int_s^t b(\sigma) \d \sigma} \nonumber\\
& \quad + \frac{\eta}{3}b(t)\mathds{1}_{[h,\infty)}(y)\int_s^t \e^{-\frac{2\eta}{3}\int_s^t b(\sigma) \d \sigma} \left[\p_t f(\sigma,s,y+h) + \p_t f(\sigma,s,y-h) \right]\d \sigma.\label{eq:ptG}
\end{align}
The derivation with respect to the second variable provides
\begin{align*}
\p_s \Gamma f(t,s,y) &= -b(s)\A f_0(y)\mathds{1}_{[h,\infty)}(y)\e^{-\frac{2\eta}{3}\int_s^t b(\sigma) \d \sigma} \\
& \qquad + \frac{\eta}{3}\mathds{1}_{[h,\infty)}(y)\int_s^t b(\sigma) \e^{-\frac{2\eta}{3}\int_s^t b(\sigma) \d \sigma} \left[\p_s f(\sigma,s,y+h) + \p_s f(\sigma,s,y-h) \right]\d \sigma.
\end{align*}
Finally, we obtain
\[
\sup_{(t,s)\in [0,T]^2}\|\Gamma f(t,s,\cdot) - \Gamma g(t,s,\cdot)\|_{\mathcal{C}^1} \leq \frac{2\eta}{3}\|b\|_\infty T\sup_{(t,s)\in [0,T]^2}\|f(t,s,\cdot)-g(t,s,\cdot)\|_{\mathcal{C}^1}
\]
so $\Gamma$ is also a contraction in $\Cr^1\left([0,T]^2;\mathcal{B}(\mathfrak{C}_x)\right)$ provided $T < \frac{3\|b\|_\infty}{2\eta}$. Iterating this result, we obtain that the solution is continuously differentiable in the two first variables on $[0,T]$ for any $T>0$.
\\
In the case the function $b$ is non increasing, we can bound $b(t)$ in~\eqref{eq:ptG} from above by $b(\sigma)$ inside the integral for all $\sigma\in [s,t]$ and obtain, taking the supremum norm
\[
|\p_t \Gamma f(t,s,y)| \leq b(t)\|\A f_0\|_{\mathcal{B}(\mathfrak{C}_x)}\e^{-\frac{2\eta}{3}\int_s^t b(\sigma) \d \sigma} + \frac{2\eta}{3}\int_s^t b(\sigma)\e^{-\frac{2\eta}{3}\int_\sigma^t b(\tau) \d \tau}\|\p_t f(\sigma,s,\cdot)\|_{\mathcal{B}(\mathfrak{C}_x)}\d\sigma
\]
and a similar formula holds for the derivative in the variable $s$. Then we obtain
\begin{align*}
&\left|\Gamma f(t,s,y)\right| + \left|\p_t\Gamma f(t,s,y)\right| + \left|\p_s\Gamma f(t,s,y)\right|\\
 \leq & \left(\|f_0\|_{\mathcal{B}(\mathfrak{C}_x)} + (b(t) + b(s))\|\A f_0\|_{\mathcal{B}(\mathfrak{C}_x)}\right)\e^{-\frac{2\eta}{3}\int_s^t b(s) \d s} + \frac{2\eta}{3}\int_s^t b(\sigma)\e^{-\frac{2\eta}{3}\int_\sigma^t b(\tau) \d \tau}\|f(\sigma,s,\cdot)\|_{\Cr^1}\d\sigma
\end{align*}
so
\begin{align*}
& \sup_{(t,s)\in [0,T]^2}\|f(t,s,\cdot)\|_{\Cr^1} \leq \left(\|f_0\|_{\mathcal{B}(\mathfrak{C}_x)} + 2b(0)\|\A f_0\|_{\mathcal{B}(\mathfrak{C}_x)}\right) \\
& \qquad \Longrightarrow \qquad \sup_{(t,s)\in [0,T]^2}\|\Gamma f(t,s,\cdot)\|_{\Cr^1} \leq \left(\|f_0\|_{\mathcal{B}(\mathfrak{C}_x)} + 2b(0)\|\A f_0\|_{\mathcal{B}(\mathfrak{C}_x)}\right)
\end{align*}
and we conclude as before.
\end{proof}
We now express the solution constructed in the previous proposition as a semigroup acting on an initial distribution. To this extent, we define $M_{s,t}^{(b)}f_0(y) = f_b(s,t,y)$ for all $(s,t,y)\in [0,T]^2 \times \R_+$ and $T>0$, where $f_b$ is the unique fixed point of $\Gamma$ with initial condition $f_0$ and associated to function $b$.
\begin{coro}\label{coro:semigroup}
For any continuous function $b : \R_+ \mapsto \R$ bounded in supremum norm and $f\in\mathcal{B}(\R_+)$, the function $(s,t,x)\mapsto M_{s,t}^{(b)} f(x)$ lies in $\Cr\left(\R_+;\mathcal{B}(\R_+)\right)\cap\Cr^1\left([0,T]^2;\mathcal{B}(\R_+)\right)$ for all $T>0$. The family $\left(M_{s,t}^{(b)}\right)_{0\leq s\leq t}$ is a nonhomogeneous semigroup, \textit{i.e.} satisfies
\[
\forall t\geq \tau\geq s\geq0,\qquad M_{s,s}^{(b)} f = f, \quad \text{and} \quad M_{s,t}^{(b)} f =  M_{\tau,t}^{(b)}M_{s,\tau}^{(b)} f.
\]
In addition, it satisfies
\[
\forall t\geq s\geq 0, \qquad \p_t M_{s,t}f = b(t)\A M_{s,t}^{(b)}f = b(t)M_{s,t}^{(b)}\A f \quad \text{and} \quad \p_s M_{s,t}^{(b)}f = - b(s)\A M_{s,t}^{(b)}f.
\]
It is a positive and conservative contraction on ${\mathfrak{C}_x}$ for all $x\in[0,h)$, i.e for all $t\geq s\geq 0$,
\begin{itemize}
\item[] \qquad $f\geq0 \Rightarrow M_{s,t}^{(b)} f \geq0$,
\item[] \qquad $M_{s,t}^{(b)} \mathds{1}_{\mathfrak{C}_x} = \mathds{1}_{\mathfrak{C}_x}$,
\item[] \qquad $\left\|M_{s,t}^{(b)}f\right\|_{\mathcal{B}(\mathfrak{C}_x)} \leq \|f\|_{\mathcal{B}(\mathfrak{C}_x)}$.
\end{itemize}
\end{coro}
\begin{proof}
To prove the semigroup property $M_{s,t}^{(b)} = M_{s,\tau}^{(b)} M_{\tau,t}^{(b)}$, we first prove the commutation properties. They are trivially true on $[0,h)$, and thus we focus on the interval $[h,\infty)$. For $f\in\B(\R_+)$, we compute the derivative in $t$ of $(s,t,y) \mapsto M_{s,t}^{(b)}f(y)$
\[
\p_t M_{s,t}^{(b)}f(y) = -\frac{2\eta}{3}b(t)M_{s,t}^{(b)}f(y) + \frac{\eta}{3}b(t)\left[M_{s,t}^{(b)}f(y+h) + M_{s,t}^{(b)}f(y-h)\right] = b(t)\A M_{s,t}^{(b)}f(y).
\]
The derivative in $s$ of $M_{s,t}f$ satisfies Equation~\eqref{eq:mild_const} on $[s,T)$ for all $s\in[0,T)$ with initial condition $-b(s)\A f$, and thus
\[
\p_s M_{s,t}f = -b(s)\A M_{s,t}f.
\]
Now we prove the desired commutation for an initial condition of the form $\mathds{1}_{\{x + kh\}}$ with $x\in[0,h)$ and $k\in\N$, and the result for a general initial distribution is deduced from the linearity and continuity of these operators. For $k\geq 2$, one has
\begin{align*}
\A\mathds{1}_{\{x + kh\}}(y) &= \mathds{1}_{\{x + kh\}}(y+h) + \mathds{1}_{\{x + kh\}}(y-h) - 2\mathds{1}_{\{x + kh\}}(y) \\
&= \mathds{1}_{\{x + (k-1)h\}}(y) + \mathds{1}_{\{x + (k+1)h\}}(y) - 2\mathds{1}_{\{x + kh\}}(y),
\end{align*}
so we compute
\begin{align*}
&M_{s,t}^{(b)}\A\mathds{1}_{\{x + kh\}}(y) = M_{s,t}^{(b)}\mathds{1}_{\{x + (k-1)h\}}(y) + M_{s,t}^{(b)}\mathds{1}_{\{x + (k+1)h\}}(y) - 2M_{s,t}^{(b)}\mathds{1}_{\{x + kh\}}(y)\\
&\A M_{s,t}^{(b)}\mathds{1}_{\{x + kh\}}(y) = M_{s,t}^{(b)}\mathds{1}_{\{x + kh\}}(y+h)+M_{s,t}^{(b)}\mathds{1}_{\{x + kh\}}(y-h) - 2M_{s,t}^{(b)}\mathds{1}_{\{x + kh\}}(y)
\end{align*}
and so it is enough to show that
\[
M_{s,t}^{(b)}\mathds{1}_{\{x + (k-1)h\}}(y) + M_{s,t}^{(b)}\mathds{1}_{\{x + (k+1)h\}}(y) = M_{s,t}^{(b)}\mathds{1}_{\{x + kh\}}(y+h)+M_{s,t}^{(b)}\mathds{1}_{\{x + kh\}}(y-h).
\]
To do so, we define $u(s,t,y) = M_{s,t}^{(b)}\mathds{1}_{\{x + kh\}}(y+h)+M_{s,t}^{(b)}\mathds{1}_{\{x + kh\}}(y-h)$, and this function satisfies
\begin{align*}
u(s,t,y) &= \left(\mathds{1}_{\{x + kh\}}(y+h) + \mathds{1}_{\{x + kh\}}(y-h)\right)\e^{-\frac{2\eta}{3}\int_s^t b(\sigma)\d \sigma} \\
& \qquad +\frac{\eta}{3}\int_s^t b(\sigma)\e^{-\frac{2\eta}{3}\int_\sigma^t b(\tau)\d \tau}\left[M_{s,\sigma}\mathds{1}_{\{x + kh\}}(y+2h) + 2\mathds{1}_{\{x + kh\}}(y) + \mathds{1}_{\{x + kh\}}(y-2h)\right]\d \sigma \\
& = \left(\mathds{1}_{\{x + (k-1)h\}} + \mathds{1}_{\{x + (k+1)h\}}\right)(y)\e^{-\frac{2\eta}{3}\int_s^t b(\sigma)\d \sigma}\\
&\qquad + \frac{\eta}{3}\int_s^t b(\sigma)\e^{-\frac{2\eta}{3}\int_\sigma^t b(\tau)\d \tau}[u(s,\sigma,y+h) + u(s,\sigma,y-h)]\d \sigma
\end{align*}
so $u$ is the fixed point of $\Gamma$ with initial condition $\mathds{1}_{\{x + (k-1)h\}} + \mathds{1}_{\{x + (k+1)h\}}$, so the claimed is proved for $k\geq2$. The case $k=1$ can be proved in a similar fashion, and the case $k = 0$ is trivial.
Now to prove the semigroup property, we define $g(\tau) = M_{s,\tau}^{(b)} M_{\tau,t}^{(b)} f$ for $\tau\in[s,t]$ and prove that it is actually constant. Using the properties already proved, we can compute its derivative and obtain
\[g'(\tau) = b(\tau)\left(\A M_{s,\tau} - M_{s,\tau}\A\right)M_{\tau,t}f = 0.\]
The positivity, contraction and conservation properties follow immediately from Proposition~\ref{prop:adj}. 
\end{proof}

\section{Wellposedness\label{sec:measure}}

In this section, we take advantage of the duality approach from section~\ref{sec:dual} to prove the wellposedness of Equation~\eqref{eq:meas_const}. We provide the measure solution by constructing a left action semigroup on $\M(\Omega)$ for $\Omega \subset \R_+$, and thus heavily rely on the results of Section~\ref{sec:dual}. For $0\leq s\leq t$, a positive measure $\mu$ on $\Omega$ and any Borel set $A\subset \Omega$, we define
\[
\mu M_{s,t}^{(b)}(A) := \int_\Omega M_{s,t}^{(b)}\1_A \d \mu 
\] 
and will show that this construction provides the unique measure solution of Equation~\eqref{eq:pde} in the sense of Definition~\ref{def:sol}. First, we show that this defines a positive measure.
\begin{lem}
Let $b$ be a continuous function, bounded in supremum norm. 
For all positive measure $\mu$ on $\R_+$ and $t\geq s\geq 0$, the set function $\mu M_{s,t}^{(b)}$ lies in $\M(\R_+)$. Additionnaly, for any $f\in\mathcal{B}(\Omega)$, one has 
\begin{equation}\label{identity}
\langle\mu M_{s,t}^{(b)},f\rangle = \langle\mu,M_{s,t}^{(b)}f\rangle.
\end{equation}
\end{lem}
A detailed proof for a similar result is given in~\cite{Gabriel2022} for weigthed measures and can be adapted to the present case. Yet we provide the main elements of the proof.
\begin{proof}
From~\eqref{eq:mild_const}, the properties of the semigroup $M_{s,t}^{(b)}$ and the monotone convergence theorem, we can easily show that for all increasing sequences $(f_n)_{n\in\N}\subset \mathcal{B}(\R_+)$ that converge pointwise to $f\in\mathcal{B}(\R_+)$, one has for all $x\geq 0$ and $t\geq 0$
\[
\lim_{n\to\infty}M_{s,t}f_n(x) = M_{s,t}f(x).
\]
This ensures that if $(A_n)_{n\in \N}$ is a countable sequence of disjoint Borel sets in $\Omega$, one has
\[\mu M_{s,t}\left(\bigsqcup_{k=0}^\infty A_k\right)=\sum_{k=0}^\infty \mu M_{s,t}(A_k)\] for all $\mu \in \M_+(\Omega)$. The remaining two axioms for a positive measure come from the positivity of the semigroup and the uniqueness of the solution of Equation~\eqref{eq:mild_const}. By definition of $\mu M_{s,t}^{(b)}$, the identity $(\mu M_{s,t}^{(b)})f = \mu(M_{s,t}^{(b)}f)$ is clearly true for any simple function $f$. Since any nonnegative measurable function is the increasing pointwise limit of simple functions, it is also valid in $[0,\infty]$ for any nonnegative $f\in \mathcal{B}(\Omega)$. Decomposing $f\in \mathcal{B}(\Omega)$ as $f = f_+ - f_-$, the linearity of the semigroup ensures that the equality $(\mu M_{s,t}^{(b)})f = \mu(M_{s,t}^{(b)}f)$ is true for all $f\in\mathcal{B}(\Omega)$.
\end{proof}
For $t\geq s\geq 0$ we extend the definition of $\mu \mapsto \left(\mu M_{s,t}^{(b)}\right)_{t\geq s\geq0}$ to $\M(\R_+)$ by setting
\[
\mu M_{s,t}^{(b)} := \mu_+ M_{s,t}^{(b)} - \mu_- M_{s,t}^{(b)},
\]
and this extension clearly preserves the identity $\langle\mu M_{s,t}^{(b)},f\rangle = \langle\mu,M_{s,t}^{(b)}f\rangle$.

\begin{prop}
For any continuous function $b : \R_+ \mapsto \R$ bounded in supremum norm and $f\in\mathcal{B}(\R_+)$, there exists at least a family of Radon measures $(\mu_t^{(b)})_{t\geq 0}$ with initial condition $\mu$ satisfying
\[
\mu_t^{(b)} f = \mu f + \int_0^t \langle\mu_s^{(b)}, \left(b(s)\A f\right)\rangle\d s.
\]
\end{prop}
\begin{proof}
For $f\in\mathcal{B}(\R_+)$ and $t\geq 0$, we write
\[
M_{0,t}^{(b)}f = f + \int_0^t \p_\sigma M_{0,s}^{(b)} f \d s = f + \int_0^t b(\sigma)M_{0,s}^{(b)}\A f \d \sigma.
\]
Integrating against $\mu$, interverting integrals, using~\eqref{identity} and taking the obtained identity at $s=0$, we obtain
\[
\langle\mu M_{0,t}^{(b)},f\rangle = \langle \mu,f\rangle + \int_0^t \langle \mu M_{0,s}^{(b)},b(s)\A f\rangle \d s.
\]
so the family $t\mapsto \mu M_{0,t}^{(b)}$ satisfies the desired relation.
\end{proof}

\begin{theorem}
There exists a unique solution $(\mu_t)_{t\geq0}$ to Equation~\eqref{eq:pde} in the sense of Definition~\ref{def:sol}. Additionally, denoting $\muin$ the initial measure, one has
\begin{itemize}
\item[] $\|(\mu)\|_{[0,T]} \leq \|\muin\|$,
\item[] $\muin \geq 0 \qquad \Rightarrow \qquad \forall t\geq 0,\, \mu_t \geq 0$.
\end{itemize}
\end{theorem}

\begin{proof}
We work in the space $\Cr([0,T];\M(\R_+))$, which is a Banach space when endowed with the norm
\[
\|(\mu)\|_{[0,T]} := \sup_{t\in[0,T]} \|\mu_t\|.
\]
Let $\Pi : \Cr([0,T];\M(\R_+)) \to \Cr([0,T];\M(\R_+))$ defined by
\[
\Pi : (\mu) \mapsto (\muin M_{0,t}^{(\mu_\bullet([h,\infty)))})_{t\geq 0}
\]
with $\muin$ the initial condition in Equation~\eqref{eq:meas_const}.
We will show that this operator is a contraction. First, simple computations using Corollary~\ref{coro:semigroup} lead to
\[
\|\left(\Pi(\mu)\right)\|_{[0,T]} \leq \|\muin\|,
\]
so the operator $\Pi$ stabilizes the ball of radius $\|\muin\|_{TV}$, and a fixed point would lie in this ball. In addition, it is enough to work with families of measures lying in this ball. Let $(\mu^1)$ and $(\mu^2)$ be such families. For $f\in \mathcal{B}(\R_+)$ and $t\geq 0$, the following inequality holds
\[
\left|\left(\Pi(\mu^1)_t - \Pi(\mu^2)_t\right)f\right| \leq \|\muin\| \left| \left(M_{0,t}^{(\mu^1_\bullet([h,\infty)))} - M_{0,t}^{(\mu^2_\bullet([h,\infty)))}\right)f \right|.
\]
We compute
\begin{align*}
&\left| \left(M_{0,t}^{(\mu^1_\bullet([h,\infty)))} - M_{0,t}^{(\mu^2_\bullet([h,\infty)))}\right)f \right|  \leq \|f\|_\infty\left|\left(\e^{-\frac{2\eta}{3}\int_0^t \mu^1_s([h,\infty))\d s} - \e^{-\frac{2\eta}{3}\int_0^t \mu^2_s([h,\infty))(s)\d s}\right)\right| \\
& \qquad + \frac{2\eta}{3}\|f\|_\infty\left(\int_0^t \left|(\mu^1_s - \mu^2_s)([h,\infty))\right|\e^{-\frac{2\eta}{3}\int_s^t \mu^1_\sigma([h,\infty))\d \sigma}\d s \right.\\
&\qquad \qquad + \left.\int_0^t\left|\mu^2_s([h,\infty))\left(\e^{-\frac{2\eta}{3}\int_s^t \mu^1_\sigma([h,\infty))\d \sigma} - \e^{-\frac{2\eta}{3}\int_s^t \mu^2_\sigma([h,\infty))\d \sigma}\right)\right|\d s\right)\\
& \leq \frac{2\eta}{3}\|f\|_\infty \e^{\frac{2\eta}{3}\|\muin\| T}\left(2 + \frac{2\eta}{3}\|\muin\| \frac{T}{2}\right)T \|(\mu^1) - (\mu^2)\|_{[0,T]}
\end{align*}
from which we deduce
\[
\|\left(\Pi(\mu^1)\right) - \left(\Pi(\mu^2)\right)\|_{[0,T]} \leq \frac{2\eta}{3} \e^{\frac{2\eta}{3}\|\muin\| T}\left(2 + \frac{\eta}{3}\|\muin\| T\right)T  \|(\mu^1) - (\mu^2)\|_{[0,T]}
\]
so $\Pi$ is a contraction on $\Cr([0,T];\M(\R_+))$ for a final time $T$ small enough.
Due to the stabilization of the ball of radius $\|\muin\|$, one can iterate the fixed point on $[T,2T]$, $[2T,3T]$... changing the initial condition each time. Finally, we obtain a unique fixed point of $\Pi$, globally defined, denoted $(\mu_t)_{t\geq 0}$. This fixed point satisfies, for all $f\in\mathcal{B}(\R_+)$,
\begin{align*}
\langle\mu_t,f\rangle =  \langle \muin,M_{0,t}^{(\mu_\bullet([h,\infty)))}f\rangle &= \langle \muin,f + \int_0^t \mu_s([h,\infty))M_{0,s}^{(\mu_\bullet([h,\infty)))}\frac{\eta}{3}\A f\d s\rangle\\ &= \langle \muin,f\rangle + \int_0^t \langle \muin M_{0,s}^{(\mu_\bullet([h,\infty)))},\mu_s([h,\infty))\frac{\eta}{3}\A f\rangle\d s
\end{align*}
and hence is a measure solution in the sense of Definition~\ref{def:sol}. From now on, we simply denote $\mu_t = \muin M_{0,t}$.
To prove the uniqueness of the solution, let $(\nu)$ be another measure solution with initial condition $\muin$, thus satisfying
\[
\langle\nu_t, f\rangle = \langle\muin, f\rangle + \int_0^t \langle\nu_s, \nu_s([h,\infty))\frac{\eta}{3}\left(\A f\right)\d s\rangle.
\]
Now we compute
\begin{align*}
|(\mu_t - \nu_t)f| &\leq \int_0^t \left|\langle\mu_s - \nu_s,\mu_s([h,\infty))\frac{\eta}{3}\left(\A f\right)\rangle\right|\d s + \int_0^t \left| \langle\nu_s,\left(\mu_s([h,\infty)) - \nu_s([h,\infty))\right)\frac{\eta}{3}\A f\rangle\right|\d s \\
& \leq \frac{4\eta}{3}\|\muin\|\|f\|_\infty\int_0^t \|\mu_s - \nu_s\|\d s
\end{align*}
and we conclude using Grönwall's lemma. Finally, to prove that for a nonnegative initial measure $\mu$, the measure solution at time $t$ is also nonnegative, we write, for $f\geq 0$
\[\langle\mu M_{s,t},f\rangle = \langle\mu,M_{s,t} f\rangle \geq 0\]
by positivity of $M_{s,t}$ on $\mathcal{B}(\R_+)$. 
\end{proof}

\section{Asymptotic behaviors\label{sec:asympt}}

In the first subsection, we use a time rescaling to reduce the study of the asymptotic behavior of the nonlinear problem to one of a linear equation. In the second subsection, we study how the limit obtained, that depends heavily on $h$, behaves when this parameter vanishes.

\subsection{Large time asymptotics}
In this section, we only consider non negative initial measures such that $\muin([h,\infty)) > 0$, otherwise for all $t\geq 0$, one has
\[\|\mu_t - \muin \mathcal{P}_h\| = 0\]and the inequality~\eqref{decay} is automatically statisfied. Let us recall that $(\mu_t)_{t\geq 0}$ denotes the unique measure solution to Equation~\eqref{eq:pde}. There exists a unique solution to Equation~\eqref{eq:mild} with $b = \mu_\bullet([h,\infty))$, simply denoted $f$. In order to obtain the asymptotic behavior in time of the measure solution, we introduce a rescaling in time of the (mild) dual equation, that is linear. Denote
\[
\tilde{f}(t,y) := f(\psi(t),0,y)
\]
with $\psi$ so that $\psi(0) = 0$ and $\tilde{f}$ is solution of
\begin{equation}\label{eq:fonda}
\ddt \tilde{f}(t,y) = \left[\tilde{f}(t,y+h) + \tilde{f}(t,y-h) - 2\tilde{f}(t,y)\right]\mathds{1}_{[h,\infty)}(y),
\end{equation}
\textit{i.e.} $\psi$ satisfies
\[
\psi'(t)\frac{\eta}{3}\mu_{\psi(t)}([h,\infty)) = 1.
\]
Denoting 
\[
\theta(t) = \frac{\eta}{3}\int_0^t \mu_s([h,\infty))\d s
\]
one has
\[
\frac{\d}{\d t}\theta(\psi(t)) = 1.
\]
Since we take $\muin$ non negative, by virtue of Lemma~\ref{prop:adj} and Corollary~\ref{coro:semigroup}, $\theta$ is increasing, and thus invertible. Taking $\psi = \theta^{-1}$ provides the desired rescaling.
\\
Proposition~\ref{prop:adj} with $b(s) = \frac{3}{\eta}$ for all $s\geq 0$ ensures that there exists a unique solution to~\eqref{eq:fonda}. Denoting $f_0$ the initial condition, this solution can be expressed by mean of a (homoeneous in time) semigroup $(N_t)_{t\geq 0}$, \textit{i.e.} $\tilde{f}(t,\cdot) = N_t f_0$. Due to te rescaling relation, one has $M_{0,t}^{(\mu_\bullet([h,\infty))}f_0 = N_{\theta(t)}f_0$, thus one can deduce the asymptotic behavior of $(\mu_t)_{t\geq 0} = \left(\mu M_{0,t}\right)_{t\geq 0}$ from that of $(\mu N_t)_{t\geq 0}$. Taking advantage of its time homogeneity, one can apply Harris-type results, extracted from~\cite{Canizo2021}. To that extent, let us introduce some vocabulary.
\\
A \textit{stochastic operator} is a linear operator $N : \M(\Omega) \to \M(\Omega)$ that preserves mass and positivity. A \textit{stochastic semigroup}
on $\M(\Omega)$ is a family $(N_t)_{t\in[0,\infty)}$ of stochastic operators $N_t : \M(\Omega) \to \M(\Omega)$ such that $N_0 = Id$ and $N_t N_s = N_{t+s}$ for all $s,t \geq 0$. Finally, $(N_t)_{t\in[0,\infty)}$ is a \textit{stochastic semigroup} on $\M_V(\Omega)$ if it is a stochastic semigroup on $\M(\Omega)$ and satisfies a growth estimate
\begin{equation}\label{stoch-semigrp}
\|\mu N_t\|_V \leq C_V\e^{\omega_V t}\|\mu\|_V
\end{equation}
for all $\mu\in\M_V(\Omega)$ and all $t \geq 0$, and for some constants $C_V \geq 1, \omega_V \geq 0$. We may now state the theorem that we shall use to obtain the asymptotic behavior of the semigroup $(N_t)_{t\in[0,\infty)}$.
\begin{theorem}[5.2 from \cite{Canizo2021}]\label{thm:Harris}
Let $V : \Omega \to [1, \infty)$ be a measurable (weight) function and let $(N_t)_{t\geq 0}$ be a stochastic semigroup on $\M_V(\Omega)$. Assume that
\begin{enumerate}
\item the semigroup $(N_t)_{t\geq 0}$ satisfies the semigroup Lyapunov condition: there exists constants $\sigma, b > 0$ such that for all $t \geq 0$ and all $\mu \in \M_V(\Omega)$
\[\|\mu N_t\|_V \leq \e^{-\sigma t}\|\mu\|_V + \frac{b}{\sigma}(1-\e^{-\sigma t})\|\mu\|;\]
\item for some $T > 0$, $N_T$ satisfies the local coupling condition: there exists $0 < \gamma_H < 1$ and $A > 0$ with $b/A < \sigma$ such that
\[
\left(y,z\in \Omega, V(y) + V(z) \leq A\right) \Longrightarrow \|(\delta_y-\delta_z)N_T\|_{TV} \leq 2 \gamma_H.
\] 
\end{enumerate}
Then the semigroup has an invariant probability measure $\mu^* \in \M_V(\Omega)$ which is unique within measures of $\M_V(\Omega)$ with total mass $1$, and there exist $\lambda,C>0$ such that
\[\|\mu N_t\|_V \leq C\e^{-\lambda t}\|\mu\|_V\]
for all $\mu\in \M_V(\Omega)$ such that $\mu(\Omega) = 0$.
\end{theorem}
Now we obtain an exponential decay for $(N_t)_{t\geq 0}$ acting on meausures defined on a given class of wealth $\mathfrak{C}_x$.
\begin{prop}\label{prop:decay_mu_class}
Fix $x\in[0,h)$ and define the weight function $V$ on $\mathfrak{C}_x$ by $V(y) = 2 - \e^{-\alpha(x)y}$ with $\alpha(x) = \frac{2\log 2}{2x+h}$. Consider the action of $(N_t)_{t\geq 0}$ to $\M_V(\mathfrak{C}_x)$ and $\mathcal{B}_V(\mathfrak{C}_x)$. Then there exists $\lambda,C>0$ independant of $x$ and a unique invariant probability measure $\mu^*_x$ such that
\[
\|\mu N_t - \mu^*_x\|_V \leq C\e^{-\lambda t} \|\mu - \mu^*_x\|_V.
\]
\end{prop}
\begin{proof}
First note for all $y\geq h$, one has $V(y-h) + V(y+h) \leq 2V(y)$. Since $1\leq V(y) \leq 2$, we readily obtain that $(N_t)_{t\geq 0}$ is well-defined on $\M_V(\mathfrak{C}_x)$. In addition, we compute  for $y\geq h$
\begin{align*}
N_tV(y) &= V(y)\e^{-2t} + \int_0^t \e^{-2(t-s)}\left[V(y+h)\frac{N_s V(y+h)}{V(y+h)} + V(y-h)\frac{N_s V(y-h)}{V(y-h)}\right]\d s \\
&\leq V(y)\e^{-2t}\left[1+\int_0^t 2\e^{2s}\|N_s V\|_V\d s\right]
\end{align*}
and deduce from it and Gronwall's Lemma that for all $t\geq 0$, one has $\|N_t V\|_V \leq 1$. It is then straightforward to obtain that $(N_t)_{t\geq0}$ is a stochastic semigroup on $\M_V(\mathfrak{C}_x)$ with $C_V = 1$ and $\omega_V = 0$. Now we prove that it satisfies a Lyapunov condition. Let $\mu\in\M_V(\mathfrak{C}_x)$ and $f\in\mathcal{B}_V(\mathfrak{C}_x)$. For any $\sigma > 0$, we easily obtain that $\A V + \sigma V \leq 2\sigma$, so $\A(V-2) \leq -\sigma(V-2)$ and by positivity of the semigroup $\p_t N_t(V-2) \leq -\sigma N_t(V-2)$, so in turn $N_t(V-2) \leq \e^{-\sigma t}(V-2)$ using Gronwall's Lemma. This inequality provides
\[N_t V \leq \e^{-\sigma t}V + 2(1-\e^{-\sigma t})\] that leads to the Lyapunov condition
\[
\|\mu N_t\|_V \leq \e^{-\sigma t}\|\mu\|_V + 2(1-\e^{-\sigma t})\|\mu\|.
\]
We now turn to the local coupling condition. We can set $A = 3$ and consider the set of elements $y,z$ of $\mathfrak{C}_x$ such that $V(y) + V(z) \leq 3$. Let us explain how we chose this particular weight function, and to that extent temporarily consider a general $V(y) = 2 - \e^{-\alpha y}$ defined on $\mathfrak{C}_x$. We rewrite $y = x+k_1 h$ and $z = x + k_2 h$, and without loss of generality take $k_1 \leq k_2$. The previous condition is thus equivalent to $\e^{-\alpha k_1 h} + \e^{-\alpha k_2 h} \geq \e^{\alpha x}$. Now we tune $\alpha$ so that the aformentionned set only contains the pairs $(x,x)$ and $(x,x+h)$, corresponding to the conditions
\[2\geq \e^{\alpha x}, \qquad 1 + \e^{-\alpha h} \geq \e^{\alpha x} \qquad \text{and} \qquad 2\e^{-\alpha h} < \e^{\alpha x}.\] When $x>0$, the first and third conditions lead to $\frac{\log 2}{x+h} < \alpha \leq \frac{\log 2}{x}$. Taking $\alpha$ as the harmonic mean of these two bounds, \textit{i.e.} \[\alpha(x) = \frac{2\log 2}{2x + h},\]
fulfills these conditions, as well as the second one. In the case $x=0$, only the third is not automatically fulfilled, and it is equivalent to
\[
\alpha > \frac{\log 2}{h} 
\]
so the previous values also works. In this setting, in order to obtain the desired inequality, it is enough to study the total variation norm of $(\delta_x - \delta_{x+h})N_t = \delta_x - \delta_{x+h}N_t$. Let $f\in \mathcal{B}(\mathfrak{C}_x)$. One has
\[
N_t f(x+h) = \frac{1-\e^{-2t}}{2}f(x) + f(x+h)\e^{-2t} + \int_0^t \e^{-2(t-s)}N_sf(x+2h)\d s
\]
so
\[\|(\delta_x - \delta_{x+h})N_t\| \leq 2\frac{1+\e^{-2t}}{2}\] which gives a local coupling condition. To conclude, we note that the constants $\lambda$ and $C$ are explicitely given in term of the constants that appear in the two hypoteses of Theorem~\ref{thm:Harris}. In our case $\sigma$ is any positive real number, $b = 2\sigma$ and $A = 3$ so are independant of $x$.
\end{proof}
We now prove a dual version of the exponential decay of $(N_t)_{t\geq 0}$, seen as an operator acting on $\mathcal{B}_V(\mathfrak{C}_x)$ and identify the limit.
\begin{prop}
Fix $x\in[0,h)$ and define the weight function $V$ on $\mathfrak{C}_x$ by $V(y) = 2 - \e^{-\alpha(x)y}$ with $\alpha(x) = \frac{2\log 2}{2x + h}$. Then there exists $\lambda,C>0$ such that
\[
\|N_t f - f(x)\1_{\mathfrak{C}_x}\|_{\mathcal{B}_V(\mathfrak{C}_x)} \leq C\e^{-\lambda t} \|f - f(x)\1_{\mathfrak{C}_x}\|_{\mathcal{B}_V(\mathfrak{C}_x)}
\]
for all $f\in\mathcal{B}_V(\mathfrak{C}_x)$.
\end{prop}
\begin{proof}
Let $k\in\N$. One has $\A\mathds{1}_{\{x\}}(x+kh) = \mathds{1}_{\{x+h\}}(x+kh)$, so for all $t\geq 0$ and $k\geq 1$, $\p_t N_t\mathds{1}_{\{x\}}(x+kh) = N_t\mathds{1}_{\{x+h\}}(x+kh) \geq 0$. Since $N_t\mathds{1}_{\mathfrak{C}_x} = \mathds{1}_{\mathfrak{C}_x}$, the positivity of $N_t$ implies that $N_t\mathds{1}_{\{x\}}(x+kh)$ converges to $l_k \leq 1$ for all $k\in\N$. In particular, $l_0=1$, and the limit function satisfies \[0 \equiv \A\left(\sum_{k=0}^\infty l_k \mathds{1}_{\{x+kh\}}\right) = \sum_{k=0}^\infty l_k \A\mathds{1}_{\{x+kh\}} = \sum_{k=1}^\infty \left(l_{k+1} + l_{k-1} - 2l_k\right)\mathds{1}_{\{x+kh\}},\]
so $l_k = 1$ for all $k\in\N$, and finally $N_t\mathds{1}_{\{x\}} \longrightarrow \mathds{1}_{\mathfrak{C}_x}$ as $t\to \infty$ for any starting time $s\geq0$. Now for any $f\geq 0$ such that $f(x+kh) \leq 1$ for all $k\in\N$, one has the inequalities
\[\1_{\{x\}} \leq f \leq \1_{\mathfrak{C}_x},\] and since $N_t$ is a positive and conservative operator on $\B(\mathfrak{C}_x)$ for all $t\geq À$, we obtain that for any such $f$, $N_tf$ converges pointwise toward $\1_{\mathfrak{C}_x}$. By linearity, for any nonnegative $f$ such that $f(x+kh) \leq f(x)$ for all $k\in\N$, the function $N_t f$ converges pointwise toward $f(x)\1_{\mathfrak{C}_x}$. Finally, for any $f\in\mathcal{B}_V(\mathfrak{C}_x)$, we modify the usual decomposition to obtain
\[f = f_+ - f_- = (f_+ + \|f\|_{\mathcal{B}_V(\mathfrak{C}_x)} \1_{\{x\}}) - (f_- + \|f\|_{\mathcal{B}_V(\mathfrak{C}_x)} \1_{\{x\}}),\]
and deduce that $N_t f$ converges pointwise toward $(f_+(x) + \|f\|_\infty)\1_{\mathfrak{C}_x} - (f_-(x) + \|f\|_\infty)\1_{\mathfrak{C}_x}$, that is $f(x)\1_{\mathfrak{C}_x}$.
\\
Now taking the dual version of the inequality from Proposition~\ref{prop:decay_mu_class} applied on $\mu - \mu^*_x$ with $\mu(\mathfrak{C}_x)=1$ provides
\[\|N_t f - \langle \mu^*_x,f \rangle\|_{\mathcal{B}_V(\mathfrak{C}_x)} \leq C\e^{-\lambda t} \|f - \langle \mu^*_x,f \rangle\|_{\mathcal{B}_V(\mathfrak{C}_x)}\]
and it remains to identify the invariant measure. From this inequality and the pointwise convergence, we obtain that $\langle \mu^*_x,f \rangle = f(x)$ for all $f\in\mathcal{B}_V(\mathfrak{C}_x)$, \textit{i.e.} $\mu^*_x = \delta_x$.
\end{proof}
Of note, the idea to decompose a function $f$ as such on $\mathfrak{C}_x$ is inspired by the framework developped in~\cite{Mukhamedov2022} and previous related articles, in a first attempt to apply those results to the present problem.
\\
We are now ready to provide a result on exponential decay for the semigroup $(N_t)_{t\geq 0}$ acting on $\M(\mathfrak{C}_x)$.
\begin{prop}\label{prop:decay_mu_N}
Let $V(y) = 2 - \e^{-\tilde{\alpha}(y)y}$ with $\tilde{\alpha}(y) := \frac{2\log 2}{2 mod(y,h) + h}$. Then there exists constants $\lambda,C>0$ and a projector operator $\mathcal{P}_h$ acting on $\M_V(\R_+)$ such that
\[
\|\mu N_t - \mu\mathcal{P}_h\|_V \leq C\e^{-\lambda t} \|\mu - \mu\mathcal{P}_h\|_V
\]
for all $\mu\in\M_V(\R_+)$ with $\mathcal{P}_h$ as stated in Theorem~\ref{thm:main}.
\end{prop}
\begin{proof}
We define the projection operator $\mathcal{P}_h$ by
\[
\mathcal{P}_h:\left\{\begin{array}{l}
\mathcal{B}(\R_+) \to \mathcal{B}(\R_+) \vspace{4mm}\\
f \mapsto \sum_{x\in [0,h)}f(x)\mathds{1}_{\mathfrak{C}_x}
\end{array}\right.
\]
Using the fact that any $f\in\mathcal{B}_V(\R_+)$ can be written as
\[f = \sum_{x\in[0,h)} f_{|\mathfrak{C}_x},\]
we compute
\begin{align*}
\|N_t f - \mathcal{P}_hf\|_{\mathcal{B}_V(\R_+)} &= \sup_{x\in [0,h)} \|N_t f_{|\mathfrak{C}_x} - \mathcal{P}_hf_{|\mathfrak{C}_x}\|_{\mathcal{B}_V(\mathfrak{C}_x)}\\
&\leq \sup_{x\in [0,h)} C\e^{-\lambda t}\|f_{|\mathfrak{C}_x} - \mathcal{P}_hf_{|\mathfrak{C}_x}\|_{\mathcal{B}_V(\mathfrak{C}_x)}\\
&\leq C\e^{-\lambda t}\|f - \mathcal{P}_hf\|_{\mathcal{B}_V(\R_+)}.
\end{align*}
The decay in weighted total variation norm is proved by taking the dual estimate of this last inequality. It remains to make explicit the action of $\mathcal{P}_h$ on measures. It is clear that on the space of function, it repeats the pattern of $f$ on $[0,h)$ on each interval $[kh,(k+1)h)$. The Riesz representation theorem states that for every measure $\mu\in \M(\R_+)$, one can define a measure denoted $\mu\mathcal{P}_h$ such that for all $f\in\mathcal{B}(\R_+)$ one has $\langle \mu\mathcal{P}_h,f\rangle = \langle \mu,\mathcal{P}_h f\rangle$. Now for any measurable set $A\subset[h,\infty)$, the definition of the operator $\mathcal{P}_h$ provides
\[(\mu\mathcal{P}_h)(A) = \langle \mu, \mathcal{P}_h(\1_A)\rangle = 0\]
so $\supp \mu\mathcal{P}_h \subset [0,h)$ for any measure $\mu$. In contrast, for a measurable set $A\in[0,h)$, let us first recall that we denote
\[
A + kh = \left\{y\in\R_+\,:\, \exists x\in A,\, y = x + kh\right\}.
\]
Then one has
\[\mu\mathcal{P}_h(A) = \left\langle \mu,\sum_{x\in[0,h)} \1_A(x)\mathds{1}_{\mathfrak{C}_x} \right\rangle = \left\langle \mu,\sum_{x\in A} \mathds{1}_{\mathfrak{C}_x} \right\rangle = \left\langle \mu,\sum_{k=0}^\infty \1_{A+kh} \right\rangle = \sum_{k=0}^\infty \mu\left(A + kh\right)\]
and the claim on $\mathcal{P}_h$ is proved.
\end{proof}
We are now ready to prove the main result of this paper.
\begin{proof}[Proof of Theorem~\ref{thm:main}]
Since $\mu_t = \mu M_{0,t} = \mu N_{\theta(t)}$, Proposition~\ref{prop:decay_mu_N} states that
\[\|\mu_t - \muin \mathcal{P}_h\|_V \leq C\e^{-\lambda \theta(t)}\|\muin - \muin \mathcal{P}_h\|_V.\] Using the estimate
\[
\mu_t([h,\infty)) \geq \frac{\muin([h,\infty))}{1+\frac{\muin([h,\infty))\eta}{3}t},
\]
provided in~\cite[Proposition 3]{Duteil2021}, we obtain that
\[\theta(t) \geq \log\left(1+\frac{\eta \muin([h,\infty))}{3}t\right)\]
and then easily reach the claimed statement.
\end{proof}
We shall now give the limit $\mu \mathcal{P}_h$ on some examples. Any measure having a density of the forme $\mathds{1}_{[kh,(k+1)h)}$ with respect to the Lebesgue measure on $\R_+$ has the same limit in large time $\mathds{1}_{[0,h)}$. Now let considere $\alpha >0$ and an exponential distribution $\mu$ with parameter $\alpha$, namely a measure with density $f(y) = \alpha\e^{-\alpha y}$ with respect to the Lesbesgue measure. Then for $0\leq a < b < h$, we compute
\begin{align*}
\mu \mathcal{P}_h([a,b)) &= \sum_{k=0}^\infty \int_{a + kh}^{b+kh} \lambda \e^{-\alpha y}\d y \\
& = \sum_{k=0}^\infty \left(\e^{-\alpha(a+kh)} - \e^{-\alpha(b+kh)}\right)\\
& = \frac{\e^{-\alpha a} - \e^{-\alpha b}}{1 - \e^{-\alpha h}}
\end{align*}
and thus for all $f\in\mathcal{B}(\R_+)$, one has \[\langle\mu\mathcal{P}_h,f\rangle = \frac{\mu\left(f_{|[0,h)}\right)}{\mu([0,h))},\]
that is to say that $\mu\mathcal{P}_h$ is an exponential distribution truncated on $[0,h)$.
\begin{remark}
It is worth noticing that we cannot expect a convergence in the space $\M_{1+Id}(\R_+)$. Indeed, considere the space of Borel functions on $\R_+$ that satisfy
\[
\sup_{y\in\R_+} \frac{|f(y)|}{1+y} < \infty.
\]
Computing formally the total wealth at the limit in time $\langle\mu\mathcal{P}_h,Id\rangle$ for a nonnegative measure $\mu\in \M_{1+Id}(\R_+)$, we obtain
\[\langle\mu,\mathcal{P}_h Id\rangle = \sum_{k=0}^\infty \int_{kh}^{(k+1)h}(y-kh)\d \mu(y) = \langle\mu,Id\rangle - h\sum_{k=0}^\infty k\mu([kh,(k+1)h))\]
so if $\mu([h,\infty)) > 0$, the total wealth in large time is strictly smaller than the total wealth at any finite time. This loss of wealth forbids a convergence in the space of measure with weight $1+Id$. However, we can provide an interpretation to this equality. The term
\[h\sum_{k=0}^\infty k\mu([kh,(k+1)h))\] corresponds to the wealth lost by players who reach an absorbing state (\textit{i.e.} wealth smaller than $h$) and is kept by the infinitely small proportion of the population who are infinitely rich.
\end{remark}
\subsection{Behavior of the asymptotic measure when the exchange parameter $h$ vanishes}

In~\cite{Duteil2021}, the authors studied the behavior as $h$ vanishes of Equation~\eqref{eq:pde} under the diffusive scaling. Here, we give a short result on what happens when $h$ vanishes, once the limit in large time is taken.
\begin{prop}
Denoting $\mathcal{P}_0$ the function defined on $M(\R_+)$ by $\mu\mathcal{P}_0 = \mu(\R_+)\delta_0$, one has
\[
\normiii{\mathcal{P}_h - \mathcal{P}_0}_{BL^*} \leq h.
\]
\end{prop}
\begin{proof}
Let $f\in\Cr^1(\R_+)$ with $\|f\|_{BL(V)}\leq 1$. The function $y\mapsto \mathcal{P}_hf(y) - f(0)$ is $h-$periodic so
\[\left|\mathcal{P}_hf - \mathcal{P}_0f\right|\leq \|f\|_{BL(V)}|x-0|\leq h.\]
We deduce that for $\mu\in BL^*(V)$ one has $\|\mu\mathcal{P}_h - \mu\mathcal{P}_0\|_{BL^*(V)} \leq h\|\mu\|_{BL^*(V)}$ and the proof is easily completed from here.
\end{proof}
Of note, the rate $h$ is not optimal. Indeed, continuating our previous examples, we obtain for $\d \mu(x) = \alpha\e^{-\alpha x}\d x$ that
\[
\|\mu\mathcal{P}_h - \mu\mathcal{P}_0\|_{BL^*(V)} \leq \frac{\alpha}{1-\e^{-\alpha h}}\frac{h^2}{2} \sim_{h\to 0} \frac{h}{2},
\]
which is independant of the parameter $\lambda$. For $\d \mu(x) = \1_{[0,h)}$, we obtain a quadratic decay
\[
\|\mu\mathcal{P}_h - \mu\mathcal{P}_0\|_{BL^*(V)} \leq \frac{h^2}{2}.
\]
\section{Numerical illustrations\label{sec:num}}

In this section, we take advantage of the cases in which the limit is easily written to study the rate of convergence. In all our examples, we take $h=0.5$. We used a first order finite differences schemeto discretize the PDE, and the Simpson quadrature method to compute $\mu_t([h,\infty))$ at each step. The particular form of Equation~\eqref{eq:pde} enables to increase the time step as $\mu_t([h,\infty))$ vanishes in time. The weight function is depicted on Figure~\ref{fig:weight}.
\begin{figure}
\centering
\includegraphics[width=\textwidth]{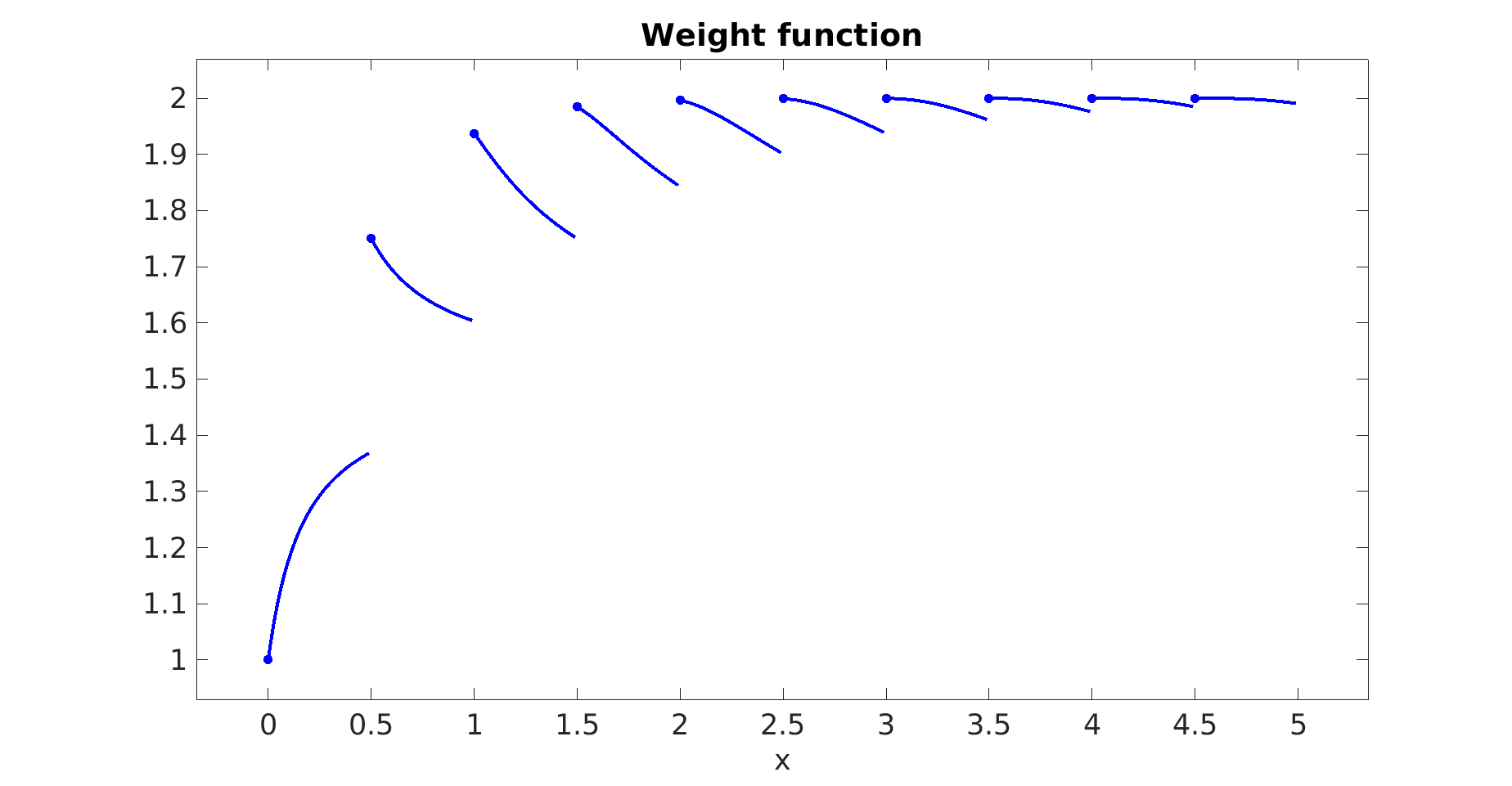}
\caption{\label{fig:weight}The discontinuous weight function $V$. The dots represent the points on which $V$ is left continuous only, \textit{i.e.} the multiples of $h$.}
\end{figure}

\subsection{Parameters $\lambda$ and $C$}

Following the lines of~\cite{Canizo2021}, we give exact values for the parameters $\lambda$ and $C$. In this paper, the authors obtain an exponential decay for a semigroup $(N_t)_{t\geq 0}$ by applying Harris' theorem to the operator $N_T$, \textit{i.e.} at a certain time $T>0$. The constants are then given by
\[
C(T) := \frac{C_V \e^{\omega_V T}}{\gamma}\frac{1+\beta}{\beta}, \qquad \lambda(T) := -\frac{\log \gamma}{T} >0,
\]
with $C_V$ and $\omega_V$ from~\eqref{stoch-semigrp} and $\beta$ and $\lambda$ obtained as follows. First, $\beta$ is the unique positive root of
\[
K\beta^2 + \left(\gamma_H - \gamma_L - K\left(1-\frac{1}{A}\right)\right)\beta + \gamma_H - 1 = 0,
\]
with $\gamma_L$ and $K$ given by the Lyapunov condition and $A$ and $\gamma_H$ by the local coupling condition. In turn, we can obtain
\[
\gamma = \max \left\{\gamma_H + \beta K, 1 - \frac{\beta}{1 + \beta}(1 - \gamma_L - K/A)\right\}
\]
and then can compute the coefficients $C$ and $\lambda$.
\\
In our case, recall first that we have $C_V = 1$ and $\omega_V = 0$. The other constants are
\begin{itemize}
\item[] $\gamma_L(T) = \e^{-\sigma T}$
\item[] $K(T) = 2(1 - \e^{-\sigma T})$
\item[] $\gamma_H(T) = \frac{1 + \e^{-2 T}}{2}$
\item[] $A = 3$
\end{itemize}
with $\sigma > 0$ and $T>0$ to be fixed. The choice $\sigma=2$ makes the computations dramatically easier, and leads to
\[
\beta = \frac{\sqrt{265} - 11}{24} \simeq 0.21995, \qquad \gamma(T) = 1 - B\left(1-\e^{-2T}\right)
\]
with $B:=\frac{\beta}{3(1+\beta)}$ and finally
\[
\lambda(T) = -\frac{\log \left(1 - B\left(1-\e^{-2T}\right)\right)}{T}.
\]
We easily see that the function $\gamma$ is decreasing, and fortunately so is $\lambda$. To see it, we first compute its derivative
\[
\lambda'(T) = \frac{1}{T^2}\left[\frac{2BT\e^{-2T}}{1 - B\left(1-\e^{-2T}\right)} + \log \left(1 - B\left(1-\e^{-2T}\right)\right)\right]
\]
so $\lambda'(T) \leq 0$ is equivalent to
\[
\frac{2BT\e^{-2T}}{1 - B\left(1-\e^{-2T}\right)} \leq -\log \left(1 - B\left(1-\e^{-2T}\right)\right).
\]
Since $-\log(1-X) \geq X$ for any $X\leq 1$, it is enough to prove that
\[
\frac{2BT\e^{-2T}}{1 - B\left(1-\e^{-2T}\right)} \leq B\left(1-\e^{-2T}\right),
\]
which is equivalent to
\[
g_1(T) := (1+2T)(1-\e^{-2T}) - B(1-\e^{-2T})^2 - 2T \geq 0.
\]
One can show that $g_1'(T) \geq 0$ is equivalent to $T \geq B(1-\e^{-2T})$. Now we set $g_2(T) := T - B(1-\e^{-2T})$. On the one hand, one has $g_1(0) = g_2(0) = 0$, and on the other
\[
g_2'(T) \geq 0\, \Longleftrightarrow\, T \geq \frac{\log 2B}{2}
\]
which is true. Thus, $g_2(T) \geq 0$ for all $T\geq 0$, and in turn $g_1(T) \geq 0$ for all $T\geq 0$, so finally $T\mapsto \lambda(T)$ is decreasing.
\\
\textit{In fine}, one can take the infimum on $T$ for $C(T)$ and the supremum for $\lambda(T)$ to obtain
\[
C := \frac{1+\beta}{\beta} = \frac{\sqrt{265} + 13}{\sqrt{265} - 11} \simeq 5.5465, \qquad \lambda := \frac{2}{3C} \simeq 0.1202.
\]
\subsection{Square functions}
The first and simpler case is square functions. We considere three cases with the same asymptotic measure, in order to compare their decay to this measure. More precisely, we take initial (probability) measures with density respectively
\begin{enumerate}
\item $\frac{1}{h}\1_{[h,2h)} = 2\1_{[0.5,1)}$
\item $\frac{1}{h}\1_{[4h,5h)} = 2\1_{[2,2.5)}$
\item $\frac{1}{h}\1_{[8h,9h)} = 2\1_{[4.5,5)}$
\end{enumerate}
with respect to the Lebesgue measure. In each case, one has $\muin([h,\infty)) = 1$, so the decay rate in the right hand side of~\eqref{decay} is equal. The very small difference that can be seen between the lines on Figure~\ref{fig:subgeom_decay_cren} is explained by the sligth difference in $\|\mu - \mu\mathcal{P}_h\|_V$ coming from the fact that the weight function $V$ is not constant.
\begin{figure}
\centering
\includegraphics[width=\textwidth]{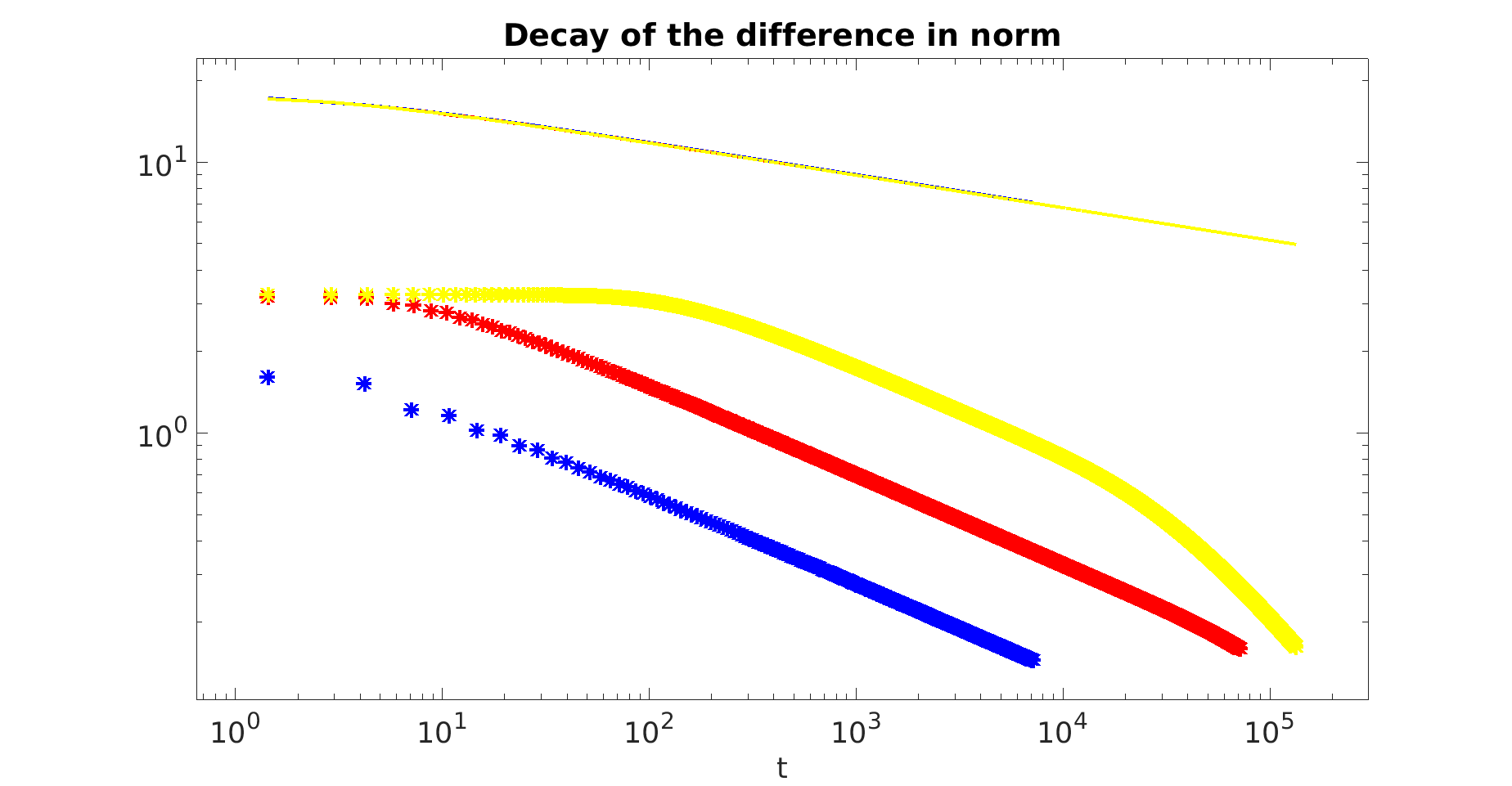}
\caption{\label{fig:subgeom_decay_cren}Loglog plot of the decay of the terms in both sides of~\eqref{decay} for square functions as initial distributions. The simulations were stopped when the quantity $\|\mu_t - \mu \mathcal{P}_h\|_V$ reached $5\%$ of its initial value. Stars: weighted norm of the difference between $\mu_t$ and $\mu\mathcal{P}_h$. Lines: decay rate times the initial difference in weighted norm. Blue: $\muin = 2\1_{[0.5,1)}$. Red: $\muin = 2\1_{[2,2.5)}$. Yellow: $\muin = 2\1_{[4.5,5)}$.}
\end{figure}
\subsection{Exponential probability distribution}

The next case is exponential probability distribution, namely initial measures with density with respect to the Lebesgue measure of the form $x\mapsto \alpha \e^{-\alpha x}$. We will take $\alpha = 0.25$, $1$ and $4$. One has $\muin([h,\infty)) = 1 - \e^{-\alpha h}$, and thus unlike the previous case of square functions we considere, this quantity depends on the initial measure. This fact is illustrated on Figure~\ref{fig:subgeom_decay_exp} on which the lines are distinct (to be compared to the ones on Figure~\ref{fig:subgeom_decay_cren}).
\begin{figure}
\centering
\includegraphics[width=\textwidth]{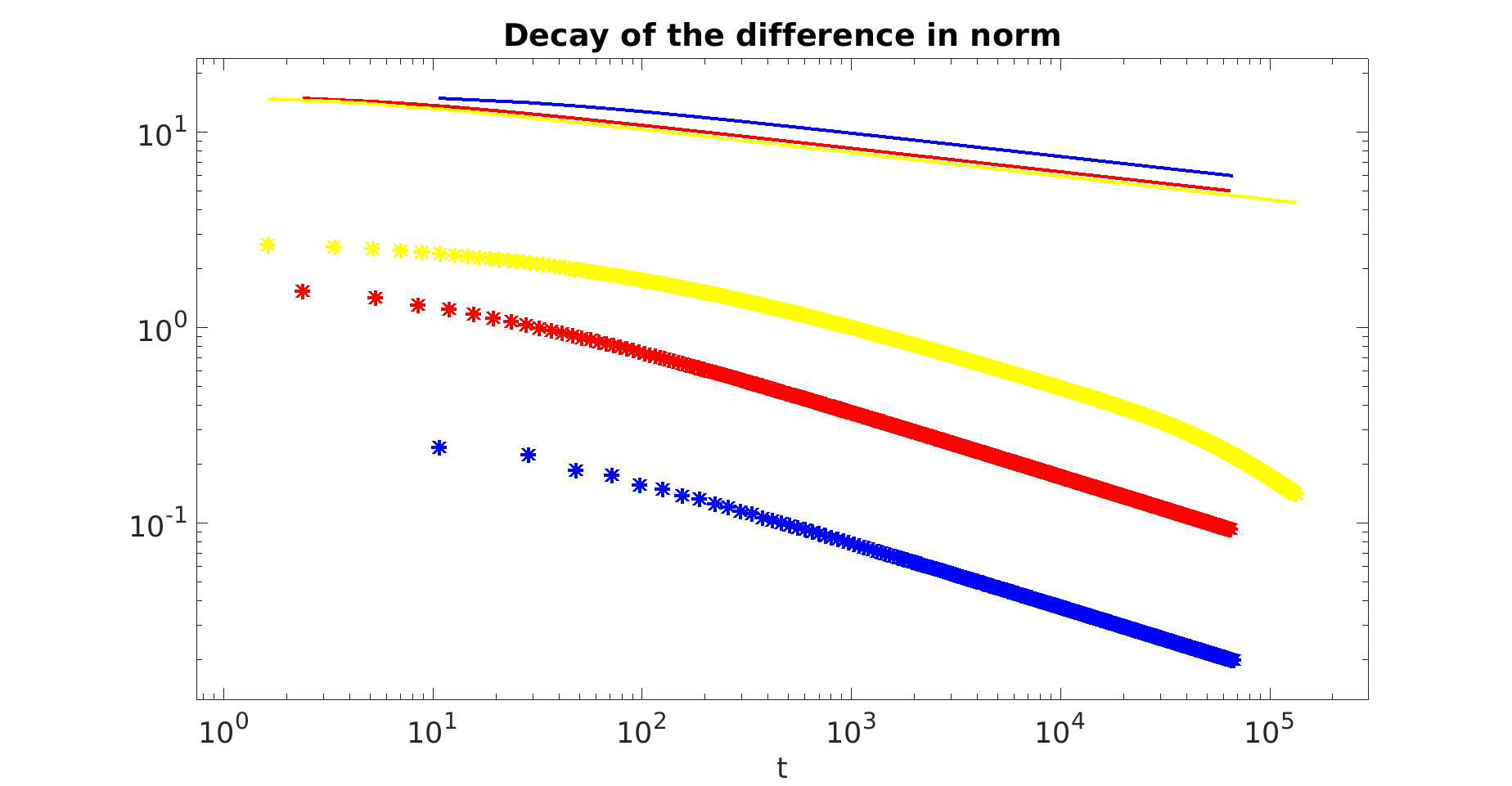}
\caption{\label{fig:subgeom_decay_exp}Loglog plot of the decay of the terms in both sides of~\eqref{decay} for exponential initial distributions. The simulations were stopped when the quantity $\|\mu_t - \mu \mathcal{P}_h\|_V$ reached $5\%$ of its initial value. Stars: weighted norm of the difference between $\mu_t$ and $\mu\mathcal{P}_h$. Lines: decay rate times the initial difference in weighted norm. Blue: $\alpha = 4$. Red: $\alpha = 1$. Yellow: $\alpha = 0.25$.}
\end{figure}

\section{Conclusion and further works\label{sec:con}}

In this work, we investigated the wellposedness and long-time behavior of a mean-field model proposed in~\cite{Duteil2021}. It describes a large population of players that meet, make transcient pairs and then possibly exchange money, depending on the result of a rock-paper-scissors game in which both players select their strategies uniformly at random. The resulting equation is a non local integrodifferential equation that is non linear and solved in a measure setting by mean of duality, after solving the adjoint equation and express its solution as a semigroup acting on an initial condition is the space of functions. Taking advantage of this duality approach, we could provide a precise description of the asymptotic behavior of the solution, that has a subgeometric decay rate with coefficients explicitely computed.
\\
A possible continuation of this paper would be to obtain optimal constants for the decay rate. Indeed, the choice $\sigma=2$ cancelled most of the dependance in $T$ in the coefficients appearing in the computations of the constants $C$ and $\lambda$. Figures~\ref{fig:subgeom_decay_cren} and~\ref{fig:subgeom_decay_exp} suggest that there is room for improvement. Other approaches to improve these constants could be finding a better lower bound for $\mu_t([h,\infty))$ or improving the constants in the local coupling condition.
\\
One might also go further in the transposition in the measure setting of the results from~\cite{Duteil2021}. Indeed, it seems unclear at the moment how the diffusive scalling introduced in the aforementionned article would interact with the definition of measure solution we adopted in  the present paper. In addition, the asymptotic behavior of the resulting equation would enable to know if the diagram depicted on Figure~\ref{fig:diag} is `commutative'. Such theoretical work con benefit from its numerical counterpart. Indeed, providing an asymptotic preserving scheme (in $h$) that is also able to capture the dynamics in large time would be alone an interesting challenge, ans in turn give clues about the expected theoretical behavior to expect.
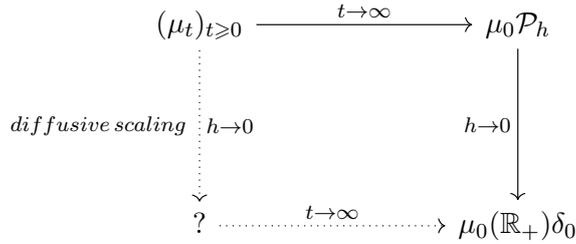
\begin{figure}
\begin{center}
\begin{tikzcd}
 (\mu_t)_{t\geq0} \arrow["t\to \infty"]{r} \arrow["diffusive\, scaling \;"',dotted]{d}{h\to 0}\& \mu_0 \mathcal{P}_h \arrow["h\to 0"']{d}{} \\
? \arrow["t\to \infty", dotted]{r} \& \mu_0 (\R_+) \delta_0
\end{tikzcd}
\end{center}
\caption{\label{fig:diag}(Commutative?) Diagram of the involved asymptotic behaviors. The top and right arrows are the subject of the present paper in the measure setting, while the left arrow was studied in~\cite{Duteil2021} in a $\mathrm{L}^1$ setting.}
\end{figure}
\\
Another way to pursue this work might be to propose and study variants on the present model. First, one might obviously change the rules to exchange money once the transcient pairs are formed. The classical \textit{matching pennies} game for instance would almost not modify the equation, leading to a mere change of a $3$ into a $2$ in the decay rate. A more interesting situation would result from such a zero-sum game in which the optimal strategy is not a uniform mixed strategy. One may also change the meeting and pair-making rules, for example imposing that a player can only initiate a game with another player that has a similar wealth, in the spirit of bounded confidence in opinion formation models~\cite{Deffuant2000,Pinasco2017}.
\\
The `concentration' effect in large time was already noted in~\cite{Duteil2021}. An interesting goal might be to add an extra term in the PDE in order to avoid it. This could be done by mean of a transport term, that might be negative for the wealthiest players (\textit{i.e.} taxes) and positive for the poorest (\textit{i.e..} social welfare). Such modification in the model would appeal for a dramatic change in the proofs, since we crucially used the fact that players remain in the same 'class of wealth' at any time. 
It seems likely that the asymptotic measure would be supported on a set larger than $[0,h)$. Particular choices of the `taxes and welfare parameters' might possibly lead to a dynamical equilibrium, in the spirit of~\cite{BDG,Gabriel2022}.
\\
Finally, a classical yet necessary complementary work would be to derive the mean-field equation from an agent based model. We might expect from such a model a dynamic similar to the one happeing in gambler's ruin type situation, except in this case a single player would own most of the money. This feature does not occure in the mean-field equation, and it would be interesting to understand where precisely it becomes impossible as the number of agents goes to infinity.

\section*{Acknowledgments}

The author is very grateful to Nastassia Pouradier Duteil for fruitful discussions on this problem and critical reading of the first manuscript.

\printbibliography

\end{document}